\documentclass{amsart}
\usepackage{amsfonts,amscd,amsthm,amsgen,amsmath,amssymb}
\usepackage[all]{xy}
\usepackage{epsfig,color}
\usepackage[vcentermath]{youngtab}
\newtheorem{theorem}{Theorem}[section]
\newtheorem{lemma}[theorem]{Lemma}
\newtheorem{corollary}[theorem]{Corollary}
\newtheorem{proposition}[theorem]{Proposition}

\theoremstyle{definition}
\newtheorem{definition}[theorem]{Definition}

\theoremstyle{remark}
\newtheorem{remark}[theorem]{Remark}

\numberwithin{equation}{section}

\DeclareMathOperator{\ind}{ind}

\begin{document}

\setlength\parskip{0.5em plus 0.1em minus 0.2em}

\title{Exotic diffeomorphisms of $4$-manifolds with $b_+ = 2$}
\author{David Baraglia}
\author{Joshua Tomlin}

\address{School of Computer and Mathematical Sciences, The University of Adelaide, Adelaide SA 5005, Australia}

\email{david.baraglia@adelaide.edu.au}
\email{joshua.tomlin@adelaide.edu.au}


\date{\today}

\begin{abstract}
Let $X$ be a compact, oriented, smooth, simply-connected $4$-manifold. The {\em mapping class group} of $X$ is defined as the group of smooth isotopy classes of diffeomorphisms of $X$. The {\em Torelli group} of $X$ is the subgroup of the mapping class group consisting of smooth isotopy classes of diffeomorphisms which are continuously isotopic to the identity. We prove that for each $n \ge 10$, the Torelli group of $2\mathbb{CP}^2 \# n \overline{\mathbb{CP}^2}$ surjects to $\mathbb{Z}^\infty$. We also prove that the mapping class group of $2 \mathbb{CP}^2 \# 10 \overline{\mathbb{CP}^2}$ is not finitely generated. Our proofs of these results makes use of Seiberg--Witten invariants for $1$-parameter familes of $4$-manifolds and in particular a gluing formula for connected sum families. Since the manifolds we consider have $b_+ = 2$, the chamber structure of the $1$-parameter Seiberg--Witten invariants plays an important role.
\end{abstract}

\maketitle




\section{Introduction}

Let $X$ be a compact, oriented, smooth, simply-connected $4$-manifold. The {\em mapping class group} $M(X)$ of $X$ is defined as the group of smooth isotopy classes of diffeomorphisms of $X$. The {\em Torelli group} $T(X) \subseteq M(X)$ is the subgroup of smooth isotopy classes of diffeomorphisms which are continuously isotopic to the identity. Non-trivial elements of $T(X)$ represent exotic diffeomorphisms in the sense that they are continuously isotopic to the identity but not smoothly isotopic. Thus, the Torelli group $T(X)$ can be seen as a measure of the difference between continuous and smooth isotopy.

It was first established by Ruberman that the Torelli group of a compact, simply-connected $4$-manifold can be non-trivial \cite{rub1,rub2}. Ruberman constructed versions of the Donaldson and Seiberg--Witten invariants for $1$-parameter families of $4$-manifolds and used these invariants to detect non-triviality of elements of the Torelli group. More specifically, his invariant can be used to show that for every $n \ge 2$ and $m \ge 10n+1$, the Torelli group of $2n \mathbb{CP}^2 \# m \overline{\mathbb{CP}^2}$, surjects to $\mathbb{Z}^\infty$ (where $\mathbb{Z}^\infty$ denotes a free abelian group of countably infinite rank). Baraglia--Konno proved a more general gluing formula for the families Seiberg--Witten invariants \cite{bk1}, which for example can be used to show a similar result for the Torelli groups of $2n \mathbb{CP}^2 \# 10n \overline{\mathbb{CP}^2}$ for each odd $n \ge 3$ and also for for $n(S^2 \times S^2) \# nK3$ for each $n \ge 1$. In all of these cases the parity of $b_+$ must be even due to a dimension constraint and there is a further constraint that $b_+ > 2$ so as to avoid the presence of any chamber structure for the families Seiberg--Witten invariants. Thus the smallest possible value of $b_+$ allowed by these methods is $b_+ = 4$. 

The main goal of this paper is to extend the methods of \cite{rub1,rub2,bk1} to the case $b_+ = 2$, using $1$-parameter Seiberg--Witten invariants in the presence of chambers. In a series of papers \cite{tom1,tom2}, the second author has proven a general gluing formula for the families Seiberg--Witten invariants which can be applied even for low values of $b_+$ where chambers are present. The results in this paper are obtained by applying this more general gluing formula to the case of $1$-parameter families. Our first main result is that the Torelli groups of certain simply-connected $4$-manifolds with $b_+ = 2$ can be infinitely generated:

\begin{theorem}\label{thm:main1}
For each $n \ge 10$, the Torelli group of $X_n = 2\mathbb{CP}^2 \# n \overline{\mathbb{CP}^2}$ admits a surjective homomorphism $\Phi : T(X_n) \to \mathbb{Z}^\infty$.
\end{theorem}

Prior to this result it had been shown by Konno--Mallick--Taniguchi \cite{kmt} that $2\mathbb{CP}^2 \# m \overline{\mathbb{CP}^2}$ has an exotic diffeomorphism for each $m \ge 11$. Qiu proves that $2\mathbb{CP}^2 \# 10 \overline{\mathbb{CP}^2}$ admits an exotic diffeomorphism with infinite order in the Torelli group \cite{qiu}. Qiu's method is similar to the one used in this paper, although the conclusion is not as strong.

Our second main result concerns infinite generation of the mapping class group $M(X)$ (by infinite generation, we mean to say that $M(X)$ is not finitely generated). Note that infinite generation of $M(X)$ does not follow from infinite generation of $T(X)$, because a finitely generated group can have infinitely generated subgroups.

\begin{theorem}\label{thm:main2}
The mapping class group of $X = 2\mathbb{CP}^2 \# 10 \overline{\mathbb{CP}^2}$ is not finitely generated. More precisely, there is an index $2$-subgroup $M_+(X) \subseteq M(X)$ and a surjective homomorphism $\Phi : M_+(X) \to \mathbb{Z}^\infty$.
\end{theorem}

Note that by Schreier's lemma \cite{ser} every finite index subgroup of a finitely generated group is finitely generated. So the infinite generation of $M_+(X)$ implies the infinite generation of $M(X)$.

The first examples of compact, simply-connected $4$-manifolds whose mapping class groups are not finitely generated were given by Baraglia \cite{bar} and Konno \cite{kon}, namely $M(X)$ is not finitely generated for $X = 2n \mathbb{CP}^2 \# 10n \overline{\mathbb{CP}^2}$ where $n \ge 3$ is odd and also for $X = n (S^2 \times S^2) \# nK3$, $n \ge 1$. Note that these are precisely the manifolds of the form $E(m) \# (S^2 \times S^2)$, $m \ge 2$. Theorem \ref{thm:main2} says the same result is true for the case $m=1$, that is, for $X = E(1) \# (S^2 \times S^2) = 2\mathbb{CP}^2 \# 10 \overline{\mathbb{CP}^2}$.

\subsection{Outline of the proofs of the main results}
To each diffeomorphism $f \in M(X)$, one can form the mapping cylinder $E(f)$. This is the $1$-parameter family of $4$-manifolds obtained from $[0,1] \times X$ by identifying the ends via $f$. If $\mathfrak{s}$ is a spin$^c$-structure which is preserved by $f$ and for which the expected dimension of the families Seiberg--Witten moduli space for $E(f)$ is zero, then one obtains a numerical invariant by counting the number of solutions of the Seiberg--Witten equations for the family $E(f)$. When $b_+ = 2$, this invariant depends on the choice of chamber and so is not strictly an invariant of $f$ alone. However under certain circumstances we find that a distinguished choice of chamber exists, and so we obtain invariants. More specifically, there are two cases that we consider:
\begin{itemize}
\item[(1)]{{\bf The constant chamber:} assume that $f \in T(X)$. Then $f$ acts trivially on $H^2(X ; \mathbb{R})$ and so the local system over $S^1$ whose fibres are $H^2$ of the fibres of $E(f)$ has trivial monodromy. This leads to a trivialisation (unique up to homotopy) of the bundle $\mathcal{H}^+ \to S^1$ whose fibres are $H^+$ of the fibres of $E(f)$. The constant chamber is the chamber which corresponds to the homotopy class of a constant section of $\mathcal{H}^+$ under the above trivialisation.}
\item[(2)]{{\bf The zero chamber:} assume that $c(\mathfrak{s})^2 \ge 0$ and $c(\mathfrak{s})$ is not torsion. Then there is a well-defined chamber corresponding to taking the self-dual $2$-form perturbation of the Seiberg--Witten equations to be zero.}
\end{itemize}

These two chambers are shown to coincide when they are both defined. Corresponding to the constant and zero chambers are families Seiberg--Witten invariants $SW^c_{X , \mathfrak{s}}(f)$ and $SW^0_{X , \mathfrak{s}}(f)$ depending only on $(X , \mathfrak{s})$ and the isotopy class of $f$. In particular, the constant chamber invariants define homomorphisms $SW^c_{X , \mathfrak{s}} : T(X) \to \mathbb{Z}$. Compactness properties of the Seiberg--Witten equations implies that for any given $f \in T(X)$, the invariants $SW^c_{X , \mathfrak{s}}(f)$ are non-zero for only finitely many spin$^c$-structures. Thus we obtain a homomorphism
\[
\Phi : T(X) \to \bigoplus_{\mathfrak{s}} \mathbb{Z}
\]
where the sum is over spin$^c$-structures for which the corresponding families Seiberg--Witten moduli space is zero dimensional. The proof of Theorem \ref{thm:main1} follows by showing that the image of $\Phi$ has infinite rank. For this we need to construct an infinite sequence of spin$^c$-structures $\{ \mathfrak{s}_n\}$ and diffeomorphisms $\{ t_n \}$ for which $SW^c_{X , \mathfrak{s}_n}( t_n ) \neq 0$. The diffeomorphisms $t_n$ are constructed in a similar fashion to \cite{rub1}, \cite{bk1}, making use of diffeomorphisms $E(1)_{p,q} \# (S^2 \times S^2) \cong E(1) \# (S^2 \times S^2)$.

\noindent{\bf Acknowledgments.} D. Baraglia was financially supported by an Australian Research Council Future Fellowship, FT230100092.

\section{$1$-parameter Seiberg--Witten invariants for $b_+ = 2$}

\subsection{Seiberg--Witten invariants for $1$-parameter families}

Let $X$ be a compact, oriented, smooth, simply-connected $4$-manifold. Let $Diff(X)$ denote the group of orientation preverving diffeomorphisms of $X$ with the $\mathcal{C}^\infty$-topology. $Diff(X)$ acts on $H^2(X;\mathbb{Z})$ by inverse pullback. We let $\Gamma(X) \subseteq Aut( H^2(X ; \mathbb{Z}))$ denote the group of automorphisms of $H^2(X ; \mathbb{Z})$ that can be realised by diffeomorphisms. The mapping class group $M(X)$ of $X$ is defined as the group of smooth isotopy classes of orientation preserving diffeomorphisms of $X$, thus $M(X) = \pi_0(Diff(X))$ is the group of components of $Diff(X)$. The action of $Diff(X)$ on $H^2(X ; \mathbb{Z})$ factors though $M(X)$ giving a surjective homomorphism $M(X) \to \Gamma(X)$. The Torelli group $T(X)$ is defined to be the kernel of $M(X) \to \Gamma(X)$, so we have a short exact sequence
\[
1 \to T(X) \to M(X) \to \Gamma(X) \to 1.
\]
From work of Quinn \cite{qui}, a diffeomorphism of $X$ is continuously isotopic to the identity if and only if it acts trivially on $H^2(X ; \mathbb{Z})$. Thus $T(X)$ is the group of smooth isotopy classes of diffeomorphisms which are continuously isotopic to the identity. Non-trivial elements of $T(X)$ represent exotic diffeomorphisms in the sense that they are continuously isotopic to the identity but not smoothly isotopic.

Given $f \in Diff(X)$, we can form the mapping cylinder $E(f) = [0,1] \times X /\! \sim$ where $(0,x) \sim (1,f(x))$. The projection $\pi : E(f) \to S^1 = \mathbb{R}/\mathbb{Z}$ given by $\pi(t,x) = t$ makes $E(f)$ into a smooth $1$-parameter family of $4$-manifolds with fibres diffeomorphic to $X$ and with base space $B = S^1$. For a spin$^c$-structure $\mathfrak{s}$ on $X$, let
\[
d(\mathfrak{s}) = \frac{ c(\mathfrak{s})^2 - \sigma(X) }{4} - b_+(X) - 1
\]
denote the expected dimension of the Seiberg--Witten moduli space for $(X,\mathfrak{s})$. Let $\mathcal{S}(X)$ denote the set of isomorphism classes of spin$^c$-structures on $X$ satisfying $d(\mathfrak{s}) = -1$. Let $\iota_0 : X \to E(f)$ be the map $\iota_0(x) = (0,x)$. Thus $\iota_0$ is the inclusion map for the fibre of $E(f)$ lying over $0$. If $\mathfrak{s} \in \mathcal{S}(X)$ and $f(\mathfrak{s}) = \mathfrak{s}$, then there exists a spin$^c$-structure $\widehat{\mathfrak{s}}$ on the vertical tangent bundle of $E(f)$ satisfying $\iota_0^*\widehat{\mathfrak{s}} = \mathfrak{s}$. From the Leray--Serre spectral sequence it follows that $\iota_0^* : H^2(E(f) ; \mathbb{Z}) \to H^2(X ; \mathbb{Z})$ is injective and therefore $\widehat{\mathfrak{s}}$ is uniquely determined. Let $\Pi$ denote the set of pairs $h = (g,\eta)$, where $g$ is a Riemannian metric on $X$ and $\eta$ is a $g$-self-dual $2$-form on $X$. Note that $Diff(X)$ acts on $\Pi$ by inverse pullback.

Given a path $h : [0,1] \to \Pi$ where $h_1 = f(h_0)$, we can regard $h_t = (g_t , \eta_t)$ as a family of fibrewise metrics $\{ g_t \}$ and fibrewise self-dual $2$-forms $\{ \eta_t \}$ for the family $E(f)$. Then we can consider the families Seiberg--Witten moduli space for the family $E(f)$ with respect to the families spin$^c$-structure $\widehat{\mathfrak{s}}$ and the family $h_t = (g_t , \eta_t)$ of metrics and $2$-form perturbations. We denote this moduli space by $\mathcal{M}( X , \mathfrak{s} , h_t )$ (recall that $\widehat{\mathfrak{s}}$ is uniquely determined by $\mathfrak{s}$, so the notation is justified).

Let $H^+(X)_g$ denote the space of $g$-harmonic self-dual $2$-forms and similarly define $H^-(X)_g$. If $\eta$ is any self-dual $2$-form, let $[\eta]_g \in H^+(X)_g$ denote the $L^2$-orthogonal projection of $\eta$ to $H^+(X)_g$. If $c \in H^2(X ; \mathbb{R})$ is any cohomology class, let $c^{+_g} \in H^+(X)_g$ denote the $g$-self-dual part of the unique $g$-harmonic representative of $c$. Note that the Hodge star $*_g$ defined by $g$ commutes with $L^2$-orthogonal projection to harmonic forms and induces an involution on $H^2(X ; \mathbb{R})$ whose $\pm 1$-eigenspaces are $H^{\pm}(X)_g$. So the map $[ \; . \;  ]^{+_g} : H^2(X ; \mathbb{R}) \to H^+(X)_g$ is just the projection map $H^2(X ; \mathbb{R}) \cong H^+(X)_g \oplus H^-(X)_g \to H^+(X)_g$. The Seiberg--Witten equations on $X$ for $\mathfrak{s}$ with respect to a metric $g$ and self-dual $2$-form perturbation $\eta$ admits reducible solutions if and only if $[\eta]_g = [ \pi c(\mathfrak{s})]^{+_g}$. The set of $(g,\eta) \in \Pi$ such that $[\eta]_g = [ \pi c(\mathfrak{s})]^{+_g}$ will be called the {\em wall}.

Let $\Pi^*_{\mathfrak{s}}$ denote the set of elements $h = (g , \eta) \in \Pi$ such that $\eta$ does note lie on the wall. The space $\Pi$ is contractible. Assuming $b_+(X) > 0$, the space $\Pi^*_{\mathfrak{s}}$ is homotopy equivalent to a sphere of dimension $b_+(X) - 1$. In the families setting, suppose we have a path $h : [0,1] \to \Pi_{\mathfrak{s}}^*$ such that $h_1 = f(h_0)$. Then the families moduli space $\mathcal{M}( X , \mathfrak{s} , h_t )$ has no reducibles. If $h$ is sufficiently generic, then $\mathcal{M}( X , \mathfrak{s} , h_t )$ is a compact manifold of dimension $d(\mathfrak{s}) + 1 = 0$. The mod $2$ families Seiberg--Witten invariant of $E(f)$ is given by counting the number of points of $\mathcal{M}( X , \mathfrak{s} , h_t )$ mod $2$.

The space of oriented, maximal positive definite subspaces of $H^2(X ; \mathbb{R})$ has two connected components. For an orientation preserving diffeomorphism $f$ of $X$, we let $sgn_+(f)$ equal $1$ or $-1$ according to whether the action of $f$ on $H^2(X ; \mathbb{R})$ preserves or exchanges the two component. If $sgn_+(f) = 1$, then $\mathcal{M}( X , \mathfrak{s} , h_t )$ is oriented and we obtain an integer families Seiberg--Witten invariant by taking a signed count of the points of $\mathcal{M}( X , \mathfrak{s} , h_t )$.

If $b_+(X) > 2$, then the mod $2$ families Seiberg--Witten invariant does not depend on the choice of $h$ and thus gives an invariant $SW_{X , \mathfrak{s}}(f) \in \mathbb{Z}_2$ which depends only on $(X,\mathfrak{s})$ and the diffeomorphism $f$. If in addition we have $sgn_+(f) = 1$, then the same is true for the integer families Seiberg--Witten invariant, giving an invariant $SW_{X , \mathfrak{s} , \mathbb{Z}}(f) \in \mathbb{Z}$. If $b_+(X) = 2$, which is the case of interest in this paper, the families Seiberg--Witten invariants depend on the choice of chamber. Let $P_{\mathfrak{s}}(f)$ denote the set of continuous paths $h : [0,1] \to \Pi_{\mathfrak{s}}^*$ such that $h_1 = f(h_0)$. A chamber is a homotopy class of such paths, or equivalently, a path connected component of $P_{\mathfrak{s}}(f)$. If $h \in P_{\mathfrak{s}}(f)$, write $[h] \in \pi_0( P_{\mathfrak{s}}(f))$ for the corresponding chamber. Thus the (mod $2$) $1$-parameter families Seiberg--Witten invariants in the case $b_+(X) = 2$ takes the form of an invariant
\[
SW_{X , \mathfrak{s}}^{\phi}(f) \in \mathbb{Z}_2
\]
where $\mathfrak{s} \in \mathcal{S}(X)$, $f \in Diff(X)$, $f(\mathfrak{s}) = \mathfrak{s}$ and $\phi \in \pi_0( P_{\mathfrak{s}}(f))$ is a chamber. If in addition we have $sgn_+(f) = 1$, then the integer $1$-parameter families Seiberg--Witten invariant is defined
\[
SW_{X  ,\mathfrak{s} , \mathbb{Z}}^{\phi}(f) \in \mathbb{Z}.
\]

\subsection{Special chambers}

One of the difficulties in using families Seiberg--Witten invariants to detect exotic diffeomorphisms is that the space $P_{\mathfrak{s}}(f)$ depends on the diffeomorphism $f$ and thus we do not have a uniform way to describe chambers independently of $f$. However in certain circumstances we have naturally defined chambers, which we call the constant chamber and the zero chamber. These chambers can also be considered for higher-dimensional families, so we consider them in more generality.

Let $B$ be a compact, connected smooth manifold. By a smooth family of $4$-manifolds over $B$ we mean a smooth fibre bundle $\pi : E \to B$ whose fibres are compact, oriented, smooth $4$-manifolds. A fibrewise spin$^c$-structure on $E \to B$ is a spin$^c$-structure $\widehat{\mathfrak{s}}$ on the vertical tangent bundle $T(E/B) = Ker( \pi_* : TE \to TB )$. Let $H^2(E/B) \to B$ denote the vector bundle whose fibre over $b \in B$ is given by $H^2( X_b ; \mathbb{R})$, where $X_b = \pi^{-1}(b)$. Equivalently $H^2(E/B)$ is the local system $H^2(E/B) = R^2 \pi_* \mathbb{R}$. Let $g$ denote a metric on $T(E/B)$. We can regard $g$ as a smoothly varying family $\{ g_b \}_{b \in B}$ of metrics on the fibres $\{ X_b \}_{b \in B}$ of $E$. Such a metric defines a subbundle $H^+_g(E/B) \subseteq H^2(E/B)$ whose fibre over $b \in B$ is given by $H^+_{g_b}( X_b ; \mathbb{Z})$, the subgroup of $H^2(X_b ; \mathbb{R})$ corresponding to harmonic self-dual $2$-forms. The definition of $H^+_g(E/B)$ depends on the choice of $g$, but any two choices yield isomorphic bundles. Indeed any two maximal positive definite subbundles of $H^2(E/B)$ are isomorphic. We obtain a section $w_{\widehat{\mathfrak{s}}} : B \to H^+_g(E/B)$ by setting $w_{\widehat{\mathfrak{s}}}(b) =  \pi [ c(\mathfrak{s}_b) ]^{+_{g_b}}$, where $\mathfrak{s}_b$ denotes the restriction of $\widehat{\mathfrak{s}}$ to $X_b$. The section $w_{\widehat{\mathfrak{s}}}$ should be thought of as the ``wall" defining the locus of self-dual $2$-form perturbations for which the Seiberg--Witten equations has reducible solutions. 

\begin{definition}
A {\em chamber} for $E$ with respect to $\widehat{\mathfrak{s}}$ is a connected component of the space of pairs $(g , \eta )$, where $g$ is a metric on $T(E/B)$ and $\eta$ is a family of self-dual $2$-forms such that for all $b \in B$, $[\eta_b ]_{g_b} \neq w_{\widehat{\mathfrak{s}}}(b)$. Here $b \mapsto [\eta_b]_{g_b}$ is the section of $H^+_g(E/B)$ given by taking the $g$-harmonic projection of $\eta$.
\end{definition}

Now that we have a general definition of chambers, we introduce two special cases: constant chambers and zero chambers.

\noindent {\bf Constant chambers.} Fix a basepoint $p \in B$. Assume that the monodromy representation of the family $\pi : E \to B$ on $H^2(X_p ; \mathbb{R})$ is trivial. Equivalently, the local system $R^2 \pi_* \mathbb{R}$ is trivial. Parallel translation gives a canonical trivialisation $H^2(E/B) \cong H^2(X_p ; \mathbb{R})$. Let $H \subseteq H^2(X_p ; \mathbb{R})$ be a maximal positive definite subspace. Let $H^\perp$ be the orthogonal complement of $H$ with respect to the intersection form, so $H^2(X_p ; \mathbb{R}) \cong H \oplus H^\perp$. Let $\rho_H : H^2(X_p ; \mathbb{R}) \to H$ be projection to the first factor. Then the composition with the inclusion
\[
\psi : H^+_g(E/B) \to H^2(E/B) \cong H^2(X_p ; \mathbb{R}) \buildrel \rho \over \longrightarrow H
\]
is an isomorphism from $H^+_g(E/B)$ to the trivial bundle with fibre $H$. This follows because the kernel of $\rho$ is negative definite and the intersection of a positive definite subspace and a negative definite subspace of $H^2(X_p ; \mathbb{R})$ is zero. Now let $v \in H$. Then $v$ defines a section of $H^+_g(E/B)$ by taking $\eta = \psi^{-1}(v)$. If $v$ is sufficiently large, then $v$ is disjoint from $w_{\widehat{\mathfrak{s}}}$ and so defines a chamber which we call the {\em constant chamber} (this chamber was also defined in \cite{bar2} where we called it the canonical chamber). It can be shown that this chamber does not depend on the choice of $g$, $p$ or $v$ provided $b_+(X) > 1$ (see \cite[Lemma 2.1]{bar2}). It also does not depend on the choice of $H$ since the space of maximal positive definite subspaces is connected. We use the superscript $c$ to denote families Seiberg--Witten invariants for the constant chamber.

\noindent {\bf Zero chambers.} Fix a basepoint $p \in B$ and set $X = X_p$. Set $\mathfrak{s} = \mathfrak{s}_p = \widehat{\mathfrak{s}}|_{X_p}$. Then $\mathfrak{s}$ is a spin$^c$-structure on $X$ and as $b$ varies over $B$, the pairs $(X_b , \mathfrak{s}_b)$ are all diffeomorphic to one another. Assume that $c(\mathfrak{s})^2 \ge 0$ and that $c(\mathfrak{s})$ is not torsion. Then it follows that for any family of metrics $\{ g_b \}$, we have that $[ c(\mathfrak{s}_b) ]^{+_{g_b}}$ is non-zero for all $b \in B$. Indeed since $c(\mathfrak{s}_b)^2 = c(\mathfrak{s})^2 \ge 0$, we have that $( [c(\mathfrak{s})_b]^{+_{g_b}}  )^2 \ge  ( [c(\mathfrak{s}_b)]^{-_{g_b}})^2 \ge 0$ with equality only if $[c(\mathfrak{s}_b)]^{+_{g_b}} = [c(\mathfrak{s}_b)]^{-_{g_b}} = 0$, which can only happen if $c(\mathfrak{s}_b)$ is torsion. Thus the zero perturbation $\eta = 0$ defines a chamber which we call the {\em zero chamber}. Since the only choice required in this definition was the choice of metric, it is clear that this is a well-defined chamber. We use the superscript $0$ to denote families Seiberg--Witten invariants for the zero chamber.

\begin{proposition}
The constant chamber and zero chamber agree when both are defined.
\end{proposition}
\begin{proof}
By assumption the zero and constant chambers are both defined. So $b_+(X) > 1$, the monodromy action on $H^2(X_p ; \mathbb{R})$ is trivial, $c(\mathfrak{s})^2 \ge 0$ and $c(\mathfrak{s})$ is not torsion. Since the monodromy is trivial we will identify $H^2(E/B)$ with the constant bundle $H^2(X_p ; \mathbb{R})$. Fix a maximal positive definite subspace $H \subseteq H^2(X_p ; \mathbb{R})$ and let $\rho_H : H^2(X_p ; \mathbb{R}) \to H$ be the projection. Choose a family $\{ g_b \}$ of metrics. This defines the bundle $H^+_g(E/B)$ which under the isomorphism $H^2(E/B) \cong H^2(X_p ; \mathbb{R})$ can be regarded as a family $\{ P_b \}$ of maximal positive definite subspaces of $H^2(X_p ; \mathbb{R})$. Recall the isomorphism $\psi : H^+_g(E/B) \to H$ is the composition of inclusion $H^+_g(E/B) \to H^2(E/B) = H^2(X_p ; \mathbb{R})$ with $\rho_H$. Under $\psi$, the constant chamber corresponds to a constant section $b \mapsto v \in H$, where $v$ is sufficiently large. Let $\omega : B \to H$ be the image of the wall under $\psi$, that is, $\omega(b) = \psi( w_{\widehat{\mathfrak{s}}}(b) ) = \psi( \pi c(\mathfrak{s})^{+_{g_b}})$. 

Under the isomorphism $\psi$, a chamber corresponds to a homotopy class of map $\nu : B \to H$ where $\nu(b) \neq \omega(b)$ for all $b \in B$. Equivalently, $\nu - \omega$ defines a map $B \to H \setminus \{0\}$ and two such maps $\nu_1,\nu_2$ define the same chamber if and only if $\nu_1 - \omega$ and $\nu_2 - \omega$ are homotopic as maps $B \to H \setminus \{0\}$. For the constant chamber we take a map $\nu_c : B \to H$ given by $\nu_c(b) = v$ for all $b \in B$, where $v \in H$ is sufficiently large. For the zero chamber we take $\nu_0 : B \to H$ given by $\nu_0(b) = 0$. Let $\nu_c'(b) = v - \omega(b)$ and $\nu_0'(b) = - \omega(b)$. Then $\nu'_c,\nu'_0$ are maps $B \to H \setminus \{0\}$. To show that the constant and zero chambers coincide, it suffies to show that $\nu'_c$ and $\nu'_0$ are homotopic. Since $v$ is sufficiently large, we can assume that $v \neq t \omega(b)$ for all $t \in [0,1]$ and all $b \in B$. Then $v- t\omega(b)$ defines a homotopy from $\nu'_c$ to the constant map $v$. So it remains to show that $\nu'_0$ is homotopic to a constant map. 

Let $\mathcal{G}$ be the set of maximal positive definite subspaces of $H^2(X_p ; \mathbb{R})$. If $P \in \mathcal{G}$, then we let $\rho_P : H^2(X_p ; \mathbb{R}) \to P$ denote the projection to $P$. Define a map $W : \mathcal{G} \to H \setminus \{0\}$ by setting $W(P) = \rho_H \circ \rho_P( \pi c(\mathfrak{s}) )$. Observe that $W(P) \neq 0$ because $c(\mathfrak{s})^2 \ge 0$ and $c(\mathfrak{s})$ is not torsion. Recall that the family of metrics $\{ g_b \}$ defines a family $\{ P_b \}$ of maximal positive definite subspaces of $H^2(X_p ; \mathbb{R})$. In other words, the family $\{ g_b \}$ determines a map $p : B \to \mathcal{G}$ such that $p(b) = P_b$. From the definitions of $W,p$ and $\omega$, it follows that $\omega = W \circ p$. So $\omega : B \to H \setminus \{0\}$ factors through $\mathcal{G}$.  But $\mathcal{G}$ is the space of maximal positive definite subspaces of $H^2(X_p ; \mathbb{R})$, which is a contractible space. Hence $\omega$ is homotopic to a constant map.
\end{proof}

Let us consider the zero and constant chambers in the case that $E = E(f)$ is the mapping cylinder of a diffeomorphism $f : X \to X$. We assume that $X$ is a compact, oriented, simply-connected smooth $4$-manifold with  $b_+(X) = 2$ and we assume that the spin$^c$-structure $\mathfrak{s}$ satisfies $d(\mathfrak{s}) = -1$. First consider the constant chamber. This is defined provided the monodromy representation of the family $E(f)$ on $H^2(X ; \mathbb{Z})$ is trivial. Equivalently, $f$ acts trivially on $H^2(X ; \mathbb{Z})$. Let $TDiff(X) \subseteq Diff(X)$ be the subgroup of $Diff(X)$ acting trivially on $H^2(X ; \mathbb{Z})$. Observe that the Torelli group $T(X)$ is the group of connected components of $TDiff(X)$. If $f \in TDiff(X)$, then $f$ fixes every spin$^c$-structure and so for every $f \in TDiff(X)$ and every $\mathfrak{s} \in \mathcal{S}(X)$, we have a mod $2$ invariant
\[
SW_{X , \mathfrak{s}}^c(f) \in \mathbb{Z}_2.
\]
Furthermore, since $f \in TDiff(X)$ we have $sgn_+(f) = 1$, so the integer invariants $SW_{X , \mathfrak{s} , \mathbb{Z}}^c(f) \in \mathbb{Z}$ are also defined. We can regard $SW_{X , \mathfrak{s}}^c(f)$ as a map from $TDiff(X)$ to $\mathbb{Z}_2$. In fact since the diffeomorphism class of $E(f)$ depends only on the isotopy class of $f$, it follows that this descends to a map from $T(X)$ to $\mathbb{Z}_2$:
\[
SW_{X  , \mathfrak{s}}^c : T(X) \to \mathbb{Z}_2.
\]
Similarly, the integer families Seiberg--Witten invariant defines a map
\[
SW_{X , \mathfrak{s} , \mathbb{Z}}^c : T(X) \to \mathbb{Z}.
\]
In fact, $SW_{X , \mathfrak{s}}^c$ and $SW_{X , \mathfrak{s},\mathbb{Z}}^c$ are group homomorphisms (this is essentially \cite[Lemma 2.6]{rub1}, adapted to the case $b_+(X)=2$).

Next consider the zero chamber. Recall that this is defined provided $c(\mathfrak{s})^2 > 0$, or $c(\mathfrak{s})^2 = 0$ and $c(\mathfrak{s})$ is non-torsion. In our case, since $b_+(X) = 2$ and $d(\mathfrak{s}) = -1$, it follows that $c(\mathfrak{s})^2 = 10 - b_-(X)$. Thus if $b_-(X) < 10$ then the zero chamber is defined. If $b_-(X) = 10$, then $c(\mathfrak{s})^2 = 0$ and $\sigma(X) = -8$. If $c(\mathfrak{s})$ is torsion, then it is zero (because $X$ is assumed to be simply-connected) and thus $X$ is spin, but this is impossible since $\sigma(X) = -8$ contradicts Rochlin's theorem. So under our assumptions on $X$, $c(\mathfrak{s})$ can not be torsion. Thus the zero chamber is defined whenever $b_-(X) \le 10$. In this case, we get a mod $2$ invariant
\[
SW_{X , \mathfrak{s}}^0(f) \in \mathbb{Z}_2.
\]
If in addition, $sgn_+(f) = 1$, then we get an integer invariant
\[
SW_{X , \mathfrak{s},\mathbb{Z}}^0(f) \in \mathbb{Z}.
\]
\subsection{Connected sum and blowup formulas for $1$-parameter families}

Let $X$ be a compact, oriented, simply-connected smooth $4$-manifold with $b_+(X) = 2$, let $\mathfrak{s}$ be a spin$^c$-structure with $d(\mathfrak{s}) = -1$ (equivalently, $c(\mathfrak{s})^2 = 10 - b_-(X)$) and let $f$ be a diffeomorphism of $f$ that preserves $\mathfrak{s}$. We consider two constructions of such triples $(X , \mathfrak{s} , f)$ as connected sums. We then use the general connected sum formula of \cite{tom2} to compute the families Seiberg--Witten invariants in these cases.

\noindent {\bf Connected sums with $S^2 \times S^2$.} Suppose that $X = X' \# (S^2 \times S^2)$, where $X'$ is a compact, oriented, simply-connected smooth $4$-manifold with $b_+(X') = 1$. We assume that $\mathfrak{s}$ has the form $\mathfrak{s} = \mathfrak{s}' \# \mathfrak{s}_0$, where $\mathfrak{s}'$ is a spin$^c$-structure on $X'$ with $d(\mathfrak{s}') = 0$ (equivalently, $c(\mathfrak{s}')^2 = 9 - b_-(X')$) and $\mathfrak{s}_0$ is the unique spin$^c$-structure on $S^2 \times S^2$ with $c(\mathfrak{s}_0) = 0$. Let $f',g$ be orientation preserving diffeomorphisms of $X'$ and $S^2 \times S^2$ respectively. Assume that $f',g$ each act as the identity in a neighbourhood of a point so that we can form the connected sum diffeomorphism $f = f' \# g$ on $X$. Assume further that $f'$ preserves $\mathfrak{s}'$. Then $f$ preserves $\mathfrak{s}$. We have the following gluing formula:

\begin{proposition}
Let $(X , \mathfrak{s} , f)$ be as above. Suppose that $c(\mathfrak{s})^2 \ge 0$ and that $sgn_+(g) = -1$. Then
\[
SW_{X, \mathfrak{s}}^0( f) = SW^0(X' , \mathfrak{s}') \; ({\rm mod} \; 2)
\]
where $SW^0(X',\mathfrak{s}')$ denote the ordinary Seiberg--Witten invariant of $(X' , \mathfrak{s}')$ with respect to the zero chamber.
\end{proposition}
\begin{proof}

Set $H^+ = H^+(E(f)/S^1)$, $H^+_1 = H^+(E(f')/S^1)$ and $H^+_2 = H^+(E(g)/S^1)$. Since $X = X' \# (S^2 \times S^2)$ and $f = f' \# g$, it follows that $H^+ \cong H_1^+ \oplus H_2^+$. The spin$^c$-structure $\mathfrak{s}$ is monodromy invariant and so extends to a family of spin$^c$-structures on the fibres of the vertical tangent bundle $T(E(f)/S^1)$. Let $\{A_\theta\}_{\theta \in S^1}$ be a family of spin$^c$ connections on $T(E(f)/S^1)$ and let $D = \ind D_A \in K^0(S^1)$ be the families index of $\{A_\theta\}$. Similarly, define $D_1$ and $D_2$ for $E(f')$ and $E(g)$ respectively. Then $D \cong D_1 \oplus D_2$.

In \cite{bk2} it is shown that the Seiberg--Witten invariant can be reformulated in equivariant cohomology through the Bauer-Furuta invariant in the following way. Let $\mu : V' \oplus U' \to V \oplus U$ be a finite dimensional approximation of the Seiberg--Witten monopole map for $E(f)$ where $V, V' \to S^1$ are complex vector spaces such that $V' - V = D$ and $U, U' \to S^1$ are real vector bundles such that $U = U' \oplus H^+$. We let $a$ and $a'$ denote the complex dimension of $V$ and $V'$ respectively so that $a' - a = \ind D_A$. Similarly, $b$ and $b'$ are the real dimensions of $U$ and $U'$ with $b - b' = b_+(X)$. The circle $S^1$ acts on $V, V'$ by multiplication and on $U, U'$ trivially. In the same fashion, $\mu_1 : V_1' \oplus U_1' \to V_1 \oplus U_1$ and $\mu_2 : V_2' \oplus U_2' \to V_2 \oplus U_2$ denote finite dimensional approximations for the monopole maps of $E(f')$ and $E(g)$ respectively, with $V_i' - V_i = D_i$ and $U_i = U_i' \oplus H_i^+$.

In the setting of \cite{bk2}, a chamber $\phi : S^1 \to H^+$ is a homotopy class of non-vanishing section of $H^+$. The zero chamber corresponds to the section $\phi = (-w_{\mathfrak{s}'}, -w_{\mathfrak{s}_0})$, where $w_{\mathfrak{s}'} : S^1 \to H^+_1$ is the wall for $(X', \mathfrak{s}')$ and $w_{\mathfrak{s}_0} : S^1 \to H^+_2$ is the wall for $(S^2 \times S^2, \mathfrak{s}_0)$. Note that $\phi$ is non-vanishing by our assumptions on $X$ and $\mathfrak{s}$. The assumptions $d(\mathfrak{s}) = -1$ and $d(\mathfrak{s}') = 0$ enforce $c(\mathfrak{s}')^2 \geq 0$ and hence $\phi' = -w_{\mathfrak{s}'}$ also defines a chamber for $X'$. 

Let $\pi_{\mathbb{P}(V')} : \mathbb{P}(V') \to S^1$ denote the complex projective bundle of $V'$ which has dimension $2a' - 2$. The cohomological formulation of the Seiberg--Witten invariant is the map
\begin{align*}
    \widetilde{SW}^{\phi, \mu}_{\mathbb{Z}_2} : H^{*}(\mathbb{P}(V') ; \mathbb{Z}_2) &\to H^{* - d(\mathfrak{s})}(S^1 ; \mathbb{Z}_2)\\
    \widetilde{SW}^{\phi, \mu}_{\mathbb{Z}_2}(\alpha) = (&\pi_{\mathbb{P}(V')})_* (\alpha \eta^\phi).
\end{align*}
Here $\eta^\phi \in H^{2a + b_+(X) - 1}(\mathbb{P}(V') ; \mathbb{Z}_2)$ is determined by the refined Thom class $\mu^* \phi_*(1) \in H^{2a + b}(S_{V',U'}, S_{U'})$ (see \cite[Section 3]{tom2}). If $d(\mathfrak{s}) = -1$, then the usual moduli space definition of the mod 2 families Seiberg--Witten invariant (that is, the invariant $SW^0_{X,\mathfrak{s}}(f)$) is related to the cohomological families Seiberg--Witten invariants by:
\[
SW^0_{X,\mathfrak{s}}(f) [S^1] = \widetilde{SW}^{\phi, \mu}_{\mathbb{Z}_2}(1),
\]
where $[S^1]$ is the generator of $H^1(S^1 ; \mathbb{Z}_2) \cong \mathbb{Z}_2$. The general gluing formula \cite[Theorem 1.2]{tom2} states that
\[
    \widetilde{SW}^{\phi, \mu}_{\mathbb{Z}_2}(1) = \omega(H^+_2) \widetilde{SW}_{\mathbb{Z}_2}^{\phi', \mu_1}(\omega_{S^1}(-D_2)).
\]
In this formula, it is assumed that the chamber $\phi$ is of the form $(\phi', 0)$, which holds here since $c(\mathfrak{s}_0) = 0$. Here $\omega(H^+_2)$ is the ordinary Steifel-Whitney class of $H^+_2$ and $\omega_{S^1}(-D_2) \in H_{S^1}(S^1; \mathbb{Z}_2)$ is the mod $2$ reduction of the $S^1$-equivariant Euler class of $-D_2$. Since $c(\mathfrak{s}_0) = 0$ and $\sigma(S^2 \times S^2) = 0$, the dimension of $D_2$ is zero. It follows from \cite[Theorem 3.1]{bar3} that $\omega_{S^1}(-D_2) = 1$ and
\[
    \widetilde{SW}^{\phi, \mu}_{\mathbb{Z}_2}(1) =  \omega(H^+_2) \widetilde{SW}_{\mathbb{Z}_2}^{\phi', \mu_1}(1).
\]
Since $sgn_+(g) = -1$ and $b_+(S^2 \times S^2) = 1$, $H_2^+$ is a non-orientable real line bundle. Therefore $\omega(H^+_2) = [S^1]$ and
\[
SW^0_{X,\mathfrak{s}}(f) [S^1] = \widetilde{SW}^{\phi, \mu}_{\mathbb{Z}_2}(1) = \widetilde{SW}_{\mathbb{Z}_2}^{\phi', \mu_1}(1) [S^1],
\]
from which it follows that $SW^0_{X,\mathfrak{s}}(f) = \widetilde{SW}_{\mathbb{Z}_2}^{\phi', \mu_1}(1)$. Recall that $d(\mathfrak{s}') = 0$ so $\widetilde{SW}_{\mathbb{Z}_2}^{\phi', \mu_1}(1) \in H^0(S^1 ; \mathbb{Z}_2)$. Pulling back to a point, $\widetilde{SW}_{\mathbb{Z}_2}^{\phi', \mu_1}(1)$ agrees with the unparametrised Seiberg--Witten invariant $SW^0(X', \mathfrak{s}')$, hence we have shown that $SW^0_{X,\mathfrak{s}}(f) = SW^0(X', \mathfrak{s}') \; ({\rm mod} \; 2)$.
\end{proof}

\noindent {\bf Connected sums with $\overline{\mathbb{CP}^2}$.} Suppose that $X = X' \# \overline{\mathbb{CP}^2}$, where $X'$ is a compact, oriented, simply-connected smooth $4$-manifold with $b_+(X) = 2$. We assume that $\mathfrak{s}$ has the form $\mathfrak{s} = \mathfrak{s}' \# \kappa$, where $\mathfrak{s}'$ is a spin$^c$-structure on $X'$ with $d(\mathfrak{s}') = -1$ and $\kappa$ is a spin$^c$-structure on $\overline{\mathbb{CP}^2}$ such that $c(\kappa)^2 = -1$. Let $f',g$ be orientation preserving diffeomorphisms of $X'$ and $\overline{\mathbb{CP}^2}$ respectively. Assume that $f',g$ each acts as the identity in a neighbourhood of a point so that the connected sum diffeomorphism $f = f' \# g$ is defined. Assume that $f'$ acts trivially on $H^2(X' ; \mathbb{Z})$ and that $g$ acts trivially on $H^2(\overline{\mathbb{CP}^2} ; \mathbb{Z})$. Then $f$ acts trivially on $H^2(X ; \mathbb{Z})$ and we have the following blowup formula:

\begin{proposition}
Let $(X , \mathfrak{s} , f)$ be as above. Then
\[
SW^c_{X  , \mathfrak{s} , \mathbb{Z}}(f) = SW_{X' , \mathfrak{s}' , \mathbb{Z}}^c(f').
\]
\end{proposition}
\begin{proof}
    In the same manner as above, let $H^+ = H^+(E(f)/S^1)$ and $D = \ind D_A$ with $H^+_1, H^+_2$ and $D_1, D_2$ defined similarly for $X'$ and $\overline{\mathbb{CP}^2}$. Let $\mu$ and $\mu_1$ denote finite dimensional approximations of the families Seiberg--Witten monopole map of $E(f)$ and $E(f')$ respectively. The constant chamber for $X$ defines a chamber $\phi : S^1 \to H^+$ by $\phi = c - w_{\mathfrak{s}}$, which is non-vanishing assuming $c$ is sufficiently large. Note that $b_+( \overline{\mathbb{CP}^2} ) = 0$, so $H^+_2 = 0$ and therefore we can identify $H^+$ with $H^+_2$. Under this identification, the constant chamber for $X'$ is represented by $\phi' : S^1 \to H^+$, where $\phi' = c - w_{\mathfrak{s}'}$. We have that $\phi$ and $\phi'$ are homotopic through the homotopy $\phi_t = c - t w_{\mathfrak{s}'} - (1-t)w_{\mathfrak{s}}$ (here $c$ is taken to be large enough that $\phi_t$ is non-vanishing for all $t$). Since $f$ acts trivially on $H^2(X ; \mathbb{Z})$, the integer-valued families Seiberg--Witten invariant of $E(f)$ is defined. Similarly, so is the integer-valued invariant for $E(f')$. The families Seiberg--Witten invariant $SW^c_{X , \mathfrak{s} , \mathbb{Z}}(f)$ is related to the cohomological families Seiberg--Witten invariants by
    \[
    SW^c_{X , \mathfrak{s} , \mathbb{Z}}(f) [S^1] = \widetilde{SW}^{\phi, \mu}_{\mathbb{Z}}(1),
    \]
where $[S^1]$ is the generator of $H^1(S^1 ; \mathbb{Z})$. Similarly $SW^c_{X' , \mathfrak{s}' , \mathbb{Z}}(f') [S^1] = \widetilde{SW}^{\phi', \mu'}_{\mathbb{Z}}(1)$.

The general gluing formula \cite{tom2} states that
    \[
    \widetilde{SW}^{\phi, \mu}_{\mathbb{Z}}(1) = \widetilde{SW}_{\mathbb{Z}}^{\phi', \mu_1}(e(H^+_2)e_{S^1}(-D_2)).
    \]
    Here $e(H^+_2)$ is the ordinary Euler class of $H^+_2$ and $e_{S^1}(-D_2) \in H_{S^1}(S^1; \mathbb{Z})$ is the equivariant Euler class of $-D_2$. Since $c(\kappa)^2 = -1$ and $b_{-}(\overline{\mathbb{CP}^2}) = 1$, we have that $D_2$ is dimension zero and hence $e_{S^1}(-D_2) = 1$ by \cite{bar3}. Since $H^+_2 = 0$ we have $e(H^+_2) = 1 \in H^0(S^1 ; \mathbb{Z}) = \mathbb{Z}$. Thus $SW^c_{X  , \mathfrak{s} , \mathbb{Z}}(f) = SW_{X' , \mathfrak{s}' , \mathbb{Z}}^c(f')$
\end{proof}

\begin{remark}
There is also a family blowup formula for the zero chamber, but it is of limited use since the zero chamber requires $c(\mathfrak{s})^2 \ge 0$, but each blowup decreases $c(\mathfrak{s})^2$ by $1$.
\end{remark}

\section{Seiberg--Witten invariants of K\"ahler surfaces with $b_1 = 0$ and $b_+ = 1$}

Suppose $X$ is a compact, oriented, smooth manifold with $b_1(X) = 0$ and $b_+(X) = 1$. Suppose that $g$ is a K\"ahler metric on $X$ with K\"ahler form $\omega$. Let $\mathfrak{s}_{can}$ denote the canonical spin$^c$-structure. Then there is a bijection between line bundles and spin$^c$-structures given by $L \mapsto \mathfrak{s}_L = L \otimes \mathfrak{s}_L$. The positive spinor bundle for $\mathfrak{s}_L$ can be identified with $L \oplus (L \otimes \wedge^{0,2} T^*X)$. Since $b_+(X) = 1$, there are two chambers for the Seiberg--Witten invariants of $X$. We define the {\em K\"ahler chamber} of $X$ to be the chamber determined by taking $\eta = \lambda \omega$ where $\lambda$ is a sufficiently large positive real number. We will denote this chamber by ``+". The other chamber is given by taking $\eta = \lambda \omega$ where $\lambda$ is a sufficiently large negative real number and will be denoted by ``-". The Seiberg--Witten invariants for the two chambers will be denoted by $SW^{\pm}(X , \mathfrak{s}_L)$.

Let $K$ denote the canonical bundle of $X$. The expected dimension of the Seiberg--Witten moduli space for $\mathfrak{s}_L$ is given by $d( \mathfrak{s}_L ) = L^2 - LK$. Thus if $L^2 - LK < 0$, then $SW^+(X , \mathfrak{s}_L) = SW^-(X , \mathfrak{s}_L) = 0$. If $L^2 - LK \ge 0$, then the wall-crossing formula gives
\[
SW^+(X , \mathfrak{s}_L) - SW^-(X , \mathfrak{s}_L) = 1.
\]
By charge conjugation symmetry, we also have the identity
\[
SW^{-}(X , \mathfrak{s}_L) = - SW^{+}(X , \mathfrak{s}_{KL^*}).
\]

Since $b_1(X) = 0$ and $b_+(X) = 1$, we have $H^{0,1}(X) = H^{0,2}(X) = 0$ and thus each complex line bundle admits a holomorphic structure which is unique up to isomorphism. Let $h^i(L)$ denote the dimension of $H^i(X , L)$, where $L$ is given its unique holomorphic structure. 

\begin{proposition}\label{prop:kah1}
Let $X$ be a compact K\"ahler surface with $b_1(X) = 0$ and $b_+(X) = 1$. If $L^2 - KL < 0$, then $SW^{\pm}(X , \mathfrak{s}_L) = 0$. If $L^2 - KL \ge 0$, then exactly one of $h^0(L)$ and $h^2(L)$ is positive and we have:
\begin{itemize}
\item[(1)]{If $h^0(L) > 0$, then $SW^+(X , \mathfrak{s}_L) = 1$, $SW^-(X , \mathfrak{s}_L) = 0$.}
\item[(2)]{If $h^2(L) > 0$, then $SW^+(X , \mathfrak{s}_L) = 0$, $SW^-(X , \mathfrak{s}_L) = -1$.}
\end{itemize}
\end{proposition}
\begin{proof}
We have already seen that $SW^{\pm}(X , \mathfrak{s}_L) = 0$ if $L^2 - KL < 0$. Assume that $L^2 - KL \ge 0$. Riemann--Roch gives
\[
h^0(L) - h^1(L) + h^2(L) = 1 + \frac{1}{2}( L^2 - KL ) > 0.
\]
Thus at least one of $h^0(L)$ and $h^2(L)$ is positive. If $h^0(L)$ and $h^2(L) = h^0(KL^*)$ are both positive, then there are non-zero holomorphic sections of $L$ and $KL^*$. Multiplying them together gives a non-zero holomorphic section of $K$, which is impossible because $b_+(X) = 1$ implies $h^0(K) = 0$. Thus exactly one of $h^0(L)$ and $h^2(L)$ is positive.

It is well-known \cite[Chapter 12]{sal} that for a perturbation of the form $\eta = \lambda \omega$ where $\lambda$ is positive and sufficiently large, solutions of the Seiberg--Witten equations are all irreducible and correspond to effective divisors representing $c_1(L)$. If $h^0(L) = 0$, then there are no effective divisors representing $L$ and thus $SW^+(X , \mathfrak{s}_L) = 0$. Hence $SW^-(X , \mathfrak{s}_L) = -1$ by wall-crossing. If $h^2(L) = 0$, then by Serre duality $h^0(KL^*) = 0$ and hence $SW^+(X , \mathfrak{s}_{KL^*}) = 0$. Then by charge conjugation symmetry $SW^-(X , \mathfrak{s}_L) = -SW^+(X , \mathfrak{s}_{KL^*}) = 0$. Hence $SW^+(X , \mathfrak{s}_L) = 1$, by wall-crossing.
\end{proof}

On a compact K\"ahler surface $X$ with $b_1(X) = 0$ and $b_+(X) = 1$ we have the two chambers $\pm 1$. We seek to determine when the zero perturbation lies in the $+$ or $-$ chamber, or lies on the wall. A perturbation $\eta$ lies on the wall for $\mathfrak{s}_L$ if and only if $[\eta]_g =  \pi [c(\mathfrak{s}_L)]^{+_g}$. Since the K\"ahler form $\omega$ is a non-zero self-dual harmonic $2$-form, this is equivalent to saying $\int_X \eta  \wedge \omega = \pi [c(\mathfrak{s}_L)] \cdot [\omega] = 0$. Let us define $\tau(\eta) =\int_X \eta  \wedge \omega - \pi [c(\mathfrak{s}_L)] \cdot [\omega]$. Then $\eta$ lies on the wall if and only if $\tau(\eta) = 0$. If we take $\eta = \lambda \omega$ for a sufficiently large positive real number $\lambda$, then $\tau(\eta) > 0$. It follows that if $\eta$ lies in the positive chamber, the wall, or the negative chamber according to whether $\tau(\eta)$ is positive, zero, or negative. 

Define the degree $deg_\omega(L)$ of $L$ with respect to $\omega$ to be given by
\[
deg_\omega(L) = [\omega] \cdot L.
\]
For the zero perturbation $\eta = 0$, we have 
\[
\tau(0) = -\pi [c(\mathfrak{s}_L)] \cdot [\omega] = \pi( K - 2L) \cdot [\omega] = \pi( deg_\omega(K) - 2 deg_\omega(L) ).
\]
Hence we obtain:

\begin{proposition}\label{prop:kah2}
Let $X$ be a compact K\"ahler surface with $b_1(X) = 0$ and $b_+(X) = 1$. Let $L$ be a complex line bundle on $X$. If $deg_\omega(L) < (1/2)deg_\omega(K)$, then the zero perturbation lies in the $+$ chamber. If $deg_\omega(L) > (1/2)deg_\omega(K)$, then the zero perturbation lies in the $-$ chamber and if $deg_\omega(L) = (1/2)deg_\omega(K)$, then the zero perturbation lies on the wall for $\mathfrak{s}_L$.
\end{proposition}

In general the chamber for which the zero perturbation lies might depend on the choice of K\"ahler metric. But we have seen that the zero chamber for a spin$^c$-structure $\mathfrak{s}$ is well-defined if $c(\mathfrak{s})^2 > 0$, or $c(\mathfrak{s})^2 = 0$ and $c(\mathfrak{s})$ is non-torsion. Assume that $H_1(X ; \mathbb{Z}) = 0$. If $d(\mathfrak{s}) \ge 0$ and $c(\mathfrak{s})^2 = 0$, then $c(\mathfrak{s})$ is not torsion. Indeed, if $c(\mathfrak{s})$ is torsion, then it must be zero since $H_1(X ; \mathbb{Z}) = 0$. So then $X$ is spin. Since $b_+(X) = 1$, it follows that $\sigma(X) = 0$ and then $d(\mathfrak{s}) = ( c(\mathfrak{s})^2 - \sigma(X))/4 - 1 - b_+(X) = -2 < 0$, a contradiction. Henceforth we assume that $H_1(X ; \mathbb{Z}) = 0$, $d(\mathfrak{s}) \ge 0$ and $c(\mathfrak{s})^2 \ge 0$. Then the zero chamber is a well-defined chamber. If $\mathfrak{s} = \mathfrak{s}_L$ for a line bundle $L$, then $c(\mathfrak{s}_L)^2 = (2L-K)^2 = 4(L^2 - KL) + K^2 = 4 d(\mathfrak{s}_L) + K^2$. So assuming that $d(\mathfrak{s}_L) \ge 0$, a sufficient condition for $c(\mathfrak{s}_L)^2 \ge 0$ is that $K^2 \ge 0$. Since $b_1(X) = 0$ and $b_+(X) = 1$, Noether's formula gives $K^2 = 12 - e(X) = 9 - b_-(X)$. Thus if $b_-(X) \le 9$, then $K^2 \ge 0$.

Putting Propositions \ref{prop:kah1}, \ref{prop:kah2} together, we have the following:

\begin{proposition}\label{prop:kah3}
Let $X$ be a compact K\"ahler surface with $H_1(X ; \mathbb{Z}) = 0$ and $b_+(X) = 1$. Let $L$ be a complex line bundle on $X$ with $L^2 - KL \ge 0$ and $(2L-K)^2 \ge 0$, so the zero chamber for $\mathfrak{s}_L$ is a well-defined chamber. Then we have:
\begin{itemize}
\item[(1)]{If $h^0(L) > 0$ and $0 \le deg_\omega(L) < (1/2)deg_\omega(K)$, then $SW^0(X , \mathfrak{s}_L) = 1$.}
\item[(2)]{If $h^0(L) = 0$ and $(1/2)deg_\omega(K) < deg_\omega(L) \le deg_\omega(K)$, then $SW^0(X , \mathfrak{s}_L) = -1$.}
\item[(3)]{In all other cases $SW^0(X , \mathfrak{s}_L) = 0$.}
\end{itemize}
In particular, if $SW^0(X , \mathfrak{s}_L) \neq 0$ for some $L$, then $deg_\omega(K) > 0$.
\end{proposition}

Consider the case of the canonical spin$^c$-structure $\mathfrak{s}_{can}$. Specialising to $L = \mathbb{C}$, we have that $h^0(L) = 1$ and $deg_\omega(L) = 0$. Thus
\begin{corollary}\label{cor:canzero}
Let $X$ be a compact K\"ahler surface with $H_1(X ; \mathbb{Z}) = 0$ and $b_+(X) = 1$. Assume that $K^2 \ge 0$ and $deg_\omega(K) > 0$. Then the zero chamber is well-defined and $SW^0(X , \mathfrak{s}_{can}) = 1$.
\end{corollary}

Let $E(1) = \mathbb{CP}^2 \# 9 \overline{\mathbb{CP}^2}$ be the blowup of $\mathbb{CP}^2$ at nine points given by the intersection of two generic cubic curves. Then $E(1)$ has the structure of a relatively minimal elliptic fibration with no multiple fibres \cite[Chapter 3]{gs}. Let $m,n > 1$ be coprime integers and let $X_{m,n} = E(1)_{m,n}$ be the elliptic surface obtained from $E(1)$ by performing logarithmic transformations of multiplicities $m,n$. Then $X$ is an elliptic surface of Kodaira dimension $1$ and is homeomorphic to $\mathbb{CP}^2 \# 9 \overline{\mathbb{CP}^2}$ \cite[Lemma 3.3.4]{gs}. Clearly $E(1)_{m,n}$ is relatively minimal and in fact minimal (see \cite[Page 23]{fm}). Furthermore, the divisibility of $[K]$ is exactly $mn-m-n$, that is, there is a primitive class $t \in H^2(X_{m,n} ; \mathbb{Z})$ such that $[K] = (mn-m-n)t$. In particular, if $m=2$ and $n > 1$ is odd, then $X_{m,n} = E(1)_{2,n}$ is a Dolgachev surface and the divisibility of $K$ is $2n-2-n = n-2$, which can be equal to any positive odd integer.

Let $F$ denote the class of a general fibre and $F_n,F_m$ the multiple fibres of multplicities $n,m$. Then $[F] = n[F_n] = m[F_m]$ and the canonical bundle is given by $[K] = -[F] + (n-1)[F_n] + (m-1)[F_m]$. Since $[F] = n[F_n]$, we get $deg_\omega(F_n) = (1/n)deg_\omega(F)$ and similarly $deg_\omega(F_m) = (1/m)deg_\omega(F)$. It follows that $deg_\omega(K) = deg_\omega(F)( 1 - 1/n - 1/m )$. Since $n,m > 1$ are coprime we have that $1 - 1/n - 1/m > 0$. Also $deg_\omega(F) > 0$ since $F$ is a non-zero effective divisor. Thus $deg_\omega(K) > 0$. We also have $K^2 = 0$ because $X_{m,n}$ is a minimal elliptic surface. Hence by Corollary \ref{cor:canzero}, the zero chamber is well-defined and $SW^0( X_{m,n} , \mathfrak{s}_{can}) = 1$.

\begin{proposition}
Let $X_{m,n} = E(1)_{m,n}$ where $m,n > 1$ are coprime. Let $L$ be a complex line bundle on $X_{m,n}$. If $d( \mathfrak{s}_L ) \ge 0$, then the zero chamber is well-defined. Moreover, if $SW^0(X_{m,n} , \mathfrak{s}_L ) \neq 0$, then $L = t K$ for some $t \in \mathbb{Q}$, $0 \le t \le 1$.
\end{proposition}
\begin{proof}
Assume that $d(\mathfrak{s}_L) \ge 0$, thus $L^2 - KL \ge 0$. Then $c(\mathfrak{s}_L)^2 = (2L-K)^2 = 4(L^2 - LK) \ge 0$ where we used that $K^2 = 0$. Thus the zero chamber is well-defined. Recall that exactly one of $h^0(L)$ and $h^0(KL^*)$ is non-zero. By charge conjugation symmetry $SW^0(X_{m,n} , \mathfrak{s}_L) = -SW^0(X_{m,n} , \mathfrak{s}_{KL^*})$, it suffices to consider only the case that $h^0(L) > 0$. Then from Proposition \ref{prop:kah3}, we have $0 \le deg_\omega(L) < (1/2)deg_\omega(K)$. Hence there is a unique real number $t \in [0,1/2)$ such that $deg_\omega(L) = t deg_\omega(K)$. This gives $(L-tK) \cdot \omega = 0$. Since $L - tK$ is orthogonal to $\omega$ and $b_+(X) = 1$, it follows that $L - tK \in H^-(X)$ and hence $(L-tK)^2 \le 0$, with equality if and only if $L = tK$. Expanding $(L-tK)^2$ gives
\[
0 \ge (L - tK)^2 = L^2 - 2t LK + t^2 K^2 = L^2 - 2t LK.
\]
So $L^2 \le 2t LK$ with equality if and only if $L = tK$. Combined with $L^2 - KL \ge 0$, we get $L^2 \le 2t KL \le 2t L^2$. Since $t < 1/2$, this implies that $L^2 \le 0$. On the other hand, since $Kod(X_{m,n}) \neq -\infty$ and $X_{m,n}$ is minimal, it follows that $K$ is nef \cite[III Corollary 2.4]{bhpv}, so $KL \ge 0$. Thus $L^2 \ge KL \ge 0$. So we must have $L^2 = KL = 0$. In particular, $(L-tK)^2 = 0$ and so $L = tK$.
\end{proof}

\section{Exotic diffeomorphisms}

By a theorem of Wall \cite[Page 136]{wall}, if $X$ is a simply-connected, non-spin $4$-manifold, then $X \# (S^2 \times S^2)$ is diffeomorphic to $X \# \mathbb{CP}^2 \# \overline{\mathbb{CP}^2}$. Now suppose $X$ is a simply-connected elliptic surface. Then Mandelbaum and Moishezon proved that $X \# \mathbb{CP}^2$ decomposes into a direct sum of copies of $\mathbb{CP}^2$ and $\overline{\mathbb{CP}^2}$ \cite{man}, \cite{moi}. Hence the same is true of $X \# (S^2 \times S^2)$, provided $X$ is not spin. In particular, for $X_{m,n} = E(1)_{m,n}$, we have that $X_{m,n} \# (S^2 \times S^2)$ is diffeomorphic to $2 \mathbb{CP}^2 \# 10 \overline{\mathbb{CP}^2}$.

\begin{lemma}\label{lem:torelli}
Let $X = 2 \mathbb{CP}^2 \# 10 \overline{\mathbb{CP}}^2$ and let $\mathfrak{s}$ be any spin$^c$-structure such that $c(\mathfrak{s})^2 = 0$. Then $SW_{X , \mathfrak{s} , \mathbb{Z}}^0 : T(X) \to \mathbb{Z}$ is non-zero. More specifically, $SW_{X , \mathfrak{s} , \mathbb{Z}}^0(f)$ is odd for some $f \in T(X)$.
\end{lemma}
\begin{proof}
Note that $d(\mathfrak{s}) = -1$ if and only if $c(\mathfrak{s})^2 = 0$. Thus $\mathcal{S}(X)$ can be identified with the set of characteristics which have square zero. If $f \in T(X)$ and $g \in M(X)$ then $SW^0_{X , \mathfrak{s}}(f) = sgn_+(g) SW^0_{X , g(\mathfrak{s})}( gfg^{-1})$, so it suffices to prove the result for one spin$^c$-structure in each orbit of $M(X)$ on $\mathcal{S}(X)$.

Let $L = H^2(X ; \mathbb{Z}) \cong \mathbb{Z}^{2,10}$ denote the intersection lattice of $X$ and let $Aut(L)$ denote the automorphism group of $L$. We have that $X = X' \# (S^2 \times S^2)$, where $X' = E(1) = \mathbb{CP}^2 \# 9 \overline{\mathbb{CP}^2}$. By \cite[Theorem 2]{wall}, the map $M(X) \to Aut(L)$ is surjective. Thus orbits of $M(X)$ on $\mathcal{S}(X)$ are the same as orbits of $Aut(L)$ on the set of characteristic elements of square zero. By \cite[Theorem 4]{wall0}, the set of such orbits is determined by the divisibility (i.e. $c \in L$ has divisibility $d > 0$ if $c = dx$ for a primitive class $x \in L$). Since $X$ is not spin, every characteristic element of $L$ is non-zero.

By the above observations, it suffices to prove that for each odd $d > 0$ there exists a spin$^c$-structure $\mathfrak{s}_d \in \mathcal{S}(X)$ and a diffeomorphism $t_d \in T(X)$ such that $c(\mathfrak{s}_d)$ has divisibility $d$ and $SW_{X , \mathfrak{s}_d}(t_d)$ is odd (note that the divisibility of a characteristic of $L$ must be odd). Let $L' = H^2(E(1) ; \mathbb{Z}) \cong \mathbb{Z}^{1,9}$ be the intersection lattice of $E(1)$ and let $H = H^2( S^2 \times S^2 ; \mathbb{Z})$ be the intersection lattice of $S^2 \times S^2$, so $L = L' \oplus H$. As explained above, $E(1)_{m,n} \# (S^2 \times S^2)$ is diffeomorphic to $X$ for every $m,n > 1$ coprime. In particular, for each odd $d > 0$, there is a diffeomorphism $\psi_d : E(1)_{2,d+2} \# (S^2 \times S^2) \to X$. Making use of the surjectivity of $M(X) \to Aut(L)$, we can choose $\psi_d$ such that the induced isomorphism
\[
(\psi_d^*)^{-1} : H^2( E(1)_{2,d+2} ; \mathbb{Z}) \oplus H^2( S^2 \times S^2 ; \mathbb{Z}) \to H^2(X ; \mathbb{Z}) = L' \oplus H
\]
sends $H^2(E(1)_{2,d+2})$ to $L'$ and $H^2(S^2 \times S^2 ; \mathbb{Z})$ to $H$.

Let $\rho$ be a diffeomorphism of $S^2 \times S^2$ which acts as $-1$ on $H^2(S^2 \times S^2 ; \mathbb{Z})$ and which equals the identity in a neighbourhood of some point. Define a diffeomorphism $f_0 \in M(X)$ by setting $f_0 = id_{X'} \# \rho$ and for each odd $d > 0$, define $f_d \in M(X)$ to be $f_d = \psi_q \circ ( id_{E(1)_{2,d+2}} \# \rho ) \circ \psi_q^{-1}$. Then $sgn_+(f_0) = sgn_+(f_d) = -1$ and setting $t_d  = f_d \circ f_0$, we have that $t_d \in T(X)$. Let $\mathfrak{s}_d = (\psi_q^{-1})^*(\mathfrak{s}_{can} \# \mathfrak{s}_0)$, where $\mathfrak{s}_{can}$ is the canonical spin$^c$-structure on $E(1)_{2,d+2}$ and $\mathfrak{s}_0$ is the unique spin$^c$-structure on $S^2 \times S^2$ with $c(\mathfrak{s}_0) = 0$. Set $c_d = c(\mathfrak{s}_d)$. Then $c_d^2 = 0$ and $c_d$ has divisibility $d$ (because $c_1(E(1)_{m,n})$ has divisibility $mn-m-n$). Since $c_d \in L'$, we have that $\mathfrak{s}_d = \mathfrak{s}'_d \# \mathfrak{s}_0$ for some spin$^c$-structure $\mathfrak{s}'_d$ on $X'$. The gluing formula applied to $X = X' \# (S^2 \times S^2)$ gives
\[
SW^0_{X , \mathfrak{s}_d}(f_0) = SW^0(X' , \mathfrak{s}'_d) = 0 \; ({\rm mod} \; 2).
\]
Note that $SW^0(X ' , \mathfrak{s}'_d) = 0$ because $X' = \mathbb{CP}^2 \# 9 \overline{\mathbb{CP}^2}$ has positive scalar curvature. On the other hand the gluing formula applied to $E(1)_{2,d+2} \# (S^2 \times S^2)$ gives
\[
SW^0_{X , \mathfrak{s}_d}(f_d) = SW^0( E(1)_{2,d+2} , \mathfrak{s}_{can} ) = 1 \; ({\rm mod} \; 2).
\]
Then since $f_0$ and $f_d$ preserve $\mathfrak{s}_d$, we have
\[
SW^0_{X , \mathfrak{s}_d}( t_d ) = SW^0_{X , \mathfrak{s}_d}(f_0) + SW^0_{X , \mathfrak{s}_d}(f_d) = 0+1 = 1 \; ({\rm mod} \; 2).
\]
Thus $SW^0_{X , \mathfrak{s}_d , \mathbb{Z}}(t_d)$ is odd.
\end{proof}

\begin{theorem}
For each $n \ge 10$, the Torelli group of $X_n = 2\mathbb{CP}^2 \# n \overline{\mathbb{CP}^2}$ admits a surjective homomorphism $\Phi : T(X_n) \to \mathbb{Z}^\infty$ to a free abelian group of countably infinite rank.
\end{theorem}
\begin{proof}
Let $\mathbb{Z}[\mathcal{S}(X_n)] = \bigoplus_{\mathfrak{s} \in \mathcal{S}(X_n)} \mathbb{Z}$. Elements of $\mathbb{Z}[\mathcal{S}(X_n)]$ can be regarded as functions $\mathcal{S}(X_n) \to \mathbb{Z}$ with finite support. Define a homomorphism $\Phi' : T(X_n) \to \bigoplus_{\mathfrak{s} \in \mathcal{S}(X_n) } \mathbb{Z}$ by taking $(\Phi'(f))(\mathfrak{s}) = SW^c_{X_n , \mathfrak{s} , \mathbb{Z}}(f)$. The function $\mathfrak{s} \mapsto SW^c_{X_n , \mathfrak{s} , \mathbb{Z}}(f)$ has finite support because of the compactness properties of the Seiberg--Witten equations.

The image $im(\Phi')$ of $\Phi'$ is a subgroup of the free abelian group $\mathbb{Z}[\mathcal{S}(X_n)]$ and hence is free abelian. Let $\Phi : T(X_n) \to im(\Phi')$ be the homomorphism obtained by factoring $\Phi'$ through its image. The theorem will follow if we can show that $im(\Phi')$ has infinite rank (since $\mathbb{Z}[\mathcal{S}(X_n)]$ has countably infinite rank, the rank of $im(\Phi')$ can be at most countably infinite). For $n=10$, Lemma \ref{lem:torelli} immediately implies that $im(\Phi')$ has countably infinite rank. For $n > 10$ we use induction on $n$ and the blow-up formula. Write $X_n = X_{n-1} \# \overline{\mathbb{CP}^2}$. Let $f \in T(X_{n-1})$. We can isotopy $f$ so that it is the identity in a neighbourhood of a point. Then we can form the connected sum diffeomorphism $f \# id_{\overline{\mathbb{CP}^2}} \in T(X_n)$. Let $\kappa$ be a spin$^c$-structure on $\overline{\mathbb{CP}^2}$ with $c(\kappa)^2 = -1$. Let $\mathfrak{s} \in \mathcal{S}(X_{n-1})$. Then $\mathfrak{s} \# \kappa \in \mathcal{S}(X_n)$. The blowup formula gives
\[
SW^c_{X_n , \mathfrak{s} \# \kappa  , \mathbb{Z}}(f  \# id_{\overline{\mathbb{CP}^2}}) = SW^c_{X_{n-1} , \mathfrak{s} , \mathbb{Z}}(f).
\]
Hence there are infinitely many spin$^c$-structures on $X_n$ for which $SW^c_{X_n , \mathfrak{s} , \mathbb{Z}} : T(X_n) \to \mathbb{Z}$ is non-zero and hence $\Phi'( T(X_n) )$ has infinite rank.
\end{proof}

\section{Non-finitely generated mapping class groups}

In this section we will prove that the mapping class group of $2\mathbb{CP}^2 \# 10 \overline{\mathbb{CP}^2}$ is not finitely generated. The strategy will be similar to the one used in \cite{bar}, but with some modifications to account for the chamber structure present in the $b_+ = 2$ case.

Let $X$ be a compact, oriented, smooth, simply-connected $4$-manifold with $b_+(X) = 2$. Recall that $\Pi$ denotes the set of pairs $h = (g , \eta)$ where $g$ is a Riemannian metric and $\eta$ is a $g$-self-dual $2$-form. We give $\Pi$ the $\mathcal{C}^\infty$-topology. Let $\mathcal{S}(X)$ be the set of isomorphism classes of spin$^c$-structures $\mathfrak{s}$ on $X$ with $d(\mathfrak{s}) = -1$. Let $\mathcal{M}(X,\mathfrak{s},h)$ denote the moduli space of the Seiberg--Witten equations on $X$ with respect to the spin$^c$-structure $\mathfrak{s} \in \mathcal{S}(X)$ and $h = (g,\eta) \in \Pi$. Say that $h \in \Pi$ is {\em regular} if $\mathcal{M}(X , \mathfrak{s} , h)$ is empty for every $\mathfrak{s} \in \mathcal{S}(X)$. The regular elements form a subset of $\Pi$ of Baire second category \cite{bar}. Let $\Pi^{reg} \subseteq \Pi$ denote the set of regular elements.

For $\mathfrak{s} \in \mathcal{S}(X)$, recall that $\Pi_{\mathfrak{s}}^*$ denotes the set of $h \in \Pi$ for which $\mathcal{M}(X,\mathfrak{s},h)$ contains no reducibles. Let $h_0, h_1 \in \Pi^{reg}$. Given a continuous path $h : [0,1] \to \Pi_{\mathfrak{s}}^*$ from $h_0$ to $h_1$, we can consider the families moduli space $\mathcal{M}(X,\mathfrak{s},h) = \cup_{t} \mathcal{M}(X, \mathfrak{s},h_t)$. For sufficiently generic paths $\mathcal{M}(X,\mathfrak{s} , h )$ is a compact, oriented, $0$-dimensional manifold. By a standard cobordism argument the number of points of $\mathcal{M}(X,\mathfrak{s},h)$ counted with sign depends only on the homotopy class of $h$, where homotopies are in the space of paths $[0,1] \to \Pi_{\mathfrak{s}}^*$ from $h_0$ to $h_1$. Since $b_+(X) = 2$, the space $\Pi_{\mathfrak{s}}^*$ is homotopy equivalent to a circle and thus there is not a unique homotopy class of path from $h_0$ to $h_1$. We get around this problem as follows: say that $h = (g,\eta) \in \Pi$ is {\em good (with respect to $\mathfrak{s}$)} if for all $t \in [0,1]$, we have $(g , t\eta) \in \Pi_{\mathfrak{s}}^*$. 

\begin{lemma}\label{lem:greg}
Assume that $b_-(X) \le 10$. For every Riemannian metric $g$, there exists a $g$-self-dual $2$-form $\eta$ such that $h = (g,\eta)$ is regular and is good with respect to every $\mathfrak{s} \in \mathcal{S}(X)$.
\end{lemma}
\begin{proof}
Let $g$ be any Riemannian metric. Then we have a norm on $H^2(X ; \mathbb{R})$ given by $||x||^2 = ( x^{+_g})^2 + ( x^{-_g})^2$. By identifying spin$^c$-structures with their characteristics, we regard $\mathcal{S}(X)$ as a discrete subset of $H^2(X ; \mathbb{R})$. Since $\mathcal{S}(X)$ is discrete and the unit ball $B = \{ x \in H^2(X ; \mathbb{R}) \; ||x|| \le 1 \}$ is compact, the intersection $\mathcal{S}(X) \cap B$ is finite. So there are only finitely many spin$^c$-structures $\mathfrak{s} \in \mathcal{S}(X)$ such that $|| c(\mathfrak{s}) || \le 1$. If $\mathfrak{s} \in \mathcal{S}(X)$, then $c(\mathfrak{s})^2 = 10-b_-(X)$, so $|| [c(\mathfrak{s})]^{+_g} ||^2 - || [c(\mathfrak{s}) ]^{-_g} ||^2 = 10 - b_-(X)$. Thus $|| c(\mathfrak{s}) || \le 1$ if and only if $|| [c(\mathfrak{s})]^{+_g} ||^2 \le u$, where $u = (10-b_-(X))/2 + (1/2)$. So there are finitely many spin$^c$-structures $\mathfrak{s} \in \mathcal{S}(X)$ such that $|| [ c(\mathfrak{s}) ]^{+_g} ||^2 < u$. On the other hand, since $b_-(X) \le 10$, the zero chamber is well-defined and $|| [c(\mathfrak{s}) ]^{+_g} || \neq 0$ for every $\mathfrak{s} \in \mathcal{S}(X)$. It follows that $m > 0$, where $m = inf_{\mathfrak{s} \in \mathcal{S}(X)} || [ c(\mathfrak{s}) ]^{+_g} ||$. Now if $\eta$ is any $g$-self-dual $2$-form with $|| [\eta]_g || < \pi m$, then $(g , \eta)$ is good with respect to every $\mathfrak{s} \in \mathcal{S}(X)$. But it is well known that for any given metric $g$, the set of regular perturbations for the Seiberg--Witten equations for the metric $g$ is of Baire second category. In particular, we can find a self-dual $2$-form $\eta$ such that $(g , \eta )$ is regular and $|| [\eta]_g || < \pi m$. Hence $(g,\eta)$ is regular and is good with respect to every $\mathfrak{s} \in \mathcal{S}(X)$.

\end{proof}

Let $\Pi^{good,reg}$ denote the space of $h = (g,\eta)$ which are regular and are good with respect to every $\mathfrak{s} \in \mathcal{S}(X)$. From Lemma \ref{lem:greg}, $\Pi^{good,reg}$ deformation retracts to the space of all metrics via the deformation retration $(g , \eta) \mapsto (g , (1-t)\eta)$, $t \in [0,1]$. It follows that $\Pi^{good,reg}$ is non-empty and contractible.

Now let $h_0,h_1 \in \Pi^{good,reg}$. Then for any $\mathfrak{s} \in \mathcal{S}(X)$, there is a distinguished homotopy class of path $h : [0,1] \to \Pi_{\mathfrak{s}}^*$ from $h_0$ to $h_1$, namely we require that $h_t$ is good for all $t$. We define $SW_{X , \mathfrak{s}}(h_0 , h_1 ) \in \mathbb{Z}$ to be the number of points of $\mathcal{M}(X , \mathfrak{s} , h)$ counted with sign, where $h$ is a good, sufficiently generic path from $h_0$ to $h_1$. Let $f$ be an orientation preserving diffeomorphism of $X$. Suppose that $f(\mathfrak{s}) = \mathfrak{s}$ and that $sgn_+(f) = 1$. Then for any $h_0 \in \Pi^{good,reg}$ we have an equality
\[
SW_{X , \mathfrak{s}}(h_0 , f(h_0) ) = SW_{X , \mathfrak{s} , \mathbb{Z}}^0( f ).
\]
Indeed, this follows because a path $h_t$ from $h_0$ to $f(h_0)$ defines a family of metrics and perturbations for the mapping cylinder $E(f)$ and furthermore, if $h_t$ is a path of good perturbations then it clearly defines the zero chamber. Similarly, if $f(\mathfrak{s}) = \mathfrak{s}$ and $sgn_+(f) = -1$ then we have a mod $2$ equality
\[
SW_{X , \mathfrak{s}}( h_0 , f(h_0) ) = SW_{X , \mathfrak{s}}^0(f) \; ({\rm mod} \; 2)
\]
for any $h_0 \in \Pi^{good,reg}$. Note that in the case $sgn_+(f) = -1$, the integer $SW_{X,\mathfrak{s}}(h_0 , f(h_0))$ might depend on the choice of $h_0$, but its value mod $2$ does not.

Let $\mathcal{O} \subseteq \mathcal{S}(X)$ be a subset of $\mathcal{S}(X)$ which is invariant under the action of the diffeomorphism group. Recall that $M_+(X) = \{ f \in M(X) \; | \; sgn_+(f) = 1 \}$. Following \cite{bar}, we define a homomorphism $SW_{X , \mathcal{O}} : M_+(X) \to \mathbb{Z}$ by setting
\[
SW_{X , \mathcal{O}}(f) = \sum_{\mathfrak{s} \in \mathcal{S}(X)} SW_{X , \mathfrak{s}}( h_0 , f(h_0 ) )
\]
where $h_0 \in \Pi^{good,reg}$. The fact that the right hand side is finite and does not depend on the choice of $h_0 \in \Pi^{good,reg}$ follows by the same reasoning used in \cite{bar}.

Suppose now that $b_-(X) = 10$, so that $\mathcal{S}(X)$ is the set of spin$^c$-structures $\mathfrak{s}$ satisfying $c(\mathfrak{s})^2 = 0$. Then for each odd $q > 0$, we define $\mathcal{O}_q \subset \mathcal{S}(X)$ to be the set of $\mathfrak{s} \in \mathcal{S}(X)$ such that $c(\mathfrak{s})$ has divisibility $q$, that is, $c(\mathfrak{s})$ is $q$ times a primitive element of $H^2(X ; \mathbb{Z})$. Now by adapting the arguments of \cite{bar}, we will obtain the following:

\begin{theorem}
The mapping class group of $X = 2\mathbb{CP}^2 \# 10 \overline{\mathbb{CP}^2}$ is not finitely generated. More precisely there is a surjective homomorphism $\Phi : M_+(X) \to \mathbb{Z}^\infty$.
\end{theorem}
\begin{proof}
Define $\Phi' : M_+(X) \to \mathbb{Z}^\infty$ to be given by
\[
\Phi' = \bigoplus_{n=1}^{\infty} SW_{X , \mathcal{O}_{2n-1}} : M_+(X) \to \bigoplus_{n=1}^{\infty} \mathbb{Z}.
\]
To prove the result, it suffices to show that the image of $\Phi'$ is not finitely generated. Then we let $\Phi$ be the homomorphism obtained by replacing the codomain of $\Phi'$ with the image of $\Phi'$. The argument is almost identical to the proof of \cite[Theorem 3.1]{bar}, except that now we use $E(1)_{2,2n+1}$ in place of $E(n)_q$.
\end{proof}


\bibliographystyle{amsplain}

\end{document}